\documentclass[10pt]{article}
\title{Representation embeddings and the second Brauer-Thrall conjecture}
\author{Klaus Bongartz\thanks{email:klausbongartz@online.de}\\University of Wuppertal}
\date{To my family}
\usepackage{amsthm} 

\usepackage[utf8]{inputenc}
\newtheorem{definition}{Definition}
\newtheorem{theorem}{Theorem}
\newtheorem{lemma}{Lemma}
\newtheorem{corollary}{Corollary}
\newtheorem{proposition}{Proposition}

\begin{document}
\maketitle
\begin{abstract}
 A representation embedding $F:mod\,A \rightarrow mod\,B$ between categories of finite-dimensional left modules over the $k$-algebras $A$ and $B$ is an exact functor preserving indecomposables and reflecting isomorphisms.  We show as our strongest result: For any basic representation infinite  $k$-split algebra A of dimension d there is a bimodule $_{A}M_{k[T]}$ which is free as a right module over $k[T]$  of  rank at most the maximum of 4d and 30 such that tensoring with $M$ is a representation embedding $mod\,k[T] \rightarrow mod\,A$.  The article contains other new results and  new interpretations or proofs of older material. Details are given in the following introduction.
\end{abstract}

\section{Introduction}

In section 2 we repete some  definitions and  basic properties concerning representation embeddings. The third section provides three important examples:
First we derive Gabriels generalization of the old Jans conditions on finite representation type and a result of Gelfand and Ponomarev from a new representation embedding ( see \cite{GI2,Jans,GP} ). 
Then we  
optimize Brenners famous  strict representation embedding $F$ from \cite{Brenner} and we show it induces an isomorphism between   the lattices of submodules of any module $M$ and $FM$. The third 
example constructs a representation embedding via extensions.  I observed this some years ago and it should  be already somewhere in the literature. 
 
The result of section 4 is somewhat surprising at first sight: Given any wild quiver $Q$ and any family  $(A_{i}), i \in I,$ of finitely generated algebras $A_{i}$ such that the cardinality of $I$ is  not greater than that
  of $k$ there is a family $$F_{i}: mod\,A_{i} \rightarrow mod\,kQ$$ of representation embeddings with $Hom_{kQ}(F_{i}X,F_{j}Y)=0$  for all $X,Y$ provided $i\neq j$.  The  proof requires only very few explicit matrix calculations
and elementary set theory.

Section  5 is the heart of the article with the theorem mentioned in the abstract. Its proof  depends on several deep theorems. 
Some historical remarks about this precise `numerical' version of the second Brauer-Thrall conjecture are also made.

The final section 6 contains  results about ( strict ) embedding classes and the partial orders relating them. We  determine  the minimal classes of  
representation-infinite algebras.
 In positive
 characteristic there is only the class of the classical Kronecker-algebra  whereas in characteristic 0 also the class of $k[X,Y]/(X,Y)^{2}$ is minimal.

\section{Basic material}
\subsection{Some notations and definitions}

  We are always working over an algebraically closed field $k$ and our algebras are associative finitely generated with unit i.e. quotients of the free associative algebra 
 $k\langle X_{1},X_{2}, \ldots ,X_{n}\rangle $ which is  isomorphic to the path algebra of the quiver $L_{n}$ consisting of one point and  $n$ loops. Also the  generalized Kronecker
 quiver $K_{n}$ 
consisting of two points $a,b$ and $n$ arrows from $a$ to $b$ plays an important role. For an algebra $A= kQ/I$ we identify modules with representations satisfying the relations 
imposed 
by $I$.  By  $mod \, A$  resp. $rep\,Q$ we denote the category of all finite dimensional  left $A$-modules resp. the category of all covariant finite-dimensional representations of $Q$. 
The usual duality between $mod \, A$ and  $mod \,A^{op}$ resp. between $rep\,Q$ and $rep \,Q^{op}$ is denoted by $D$.

	 \begin{definition} Let $A,B$  be some algebras as above.
	\begin{enumerate}
	\item A  representation embedding from $mod \,A$ to $mod\,B$ is a $k$-linear covariant functor $F:mod\,A \longrightarrow mod\,B$ that is exact, maps indecomposables to indecomposables 
and reflects isomorphism classes.  
	\item If $F$ as above is in addition full we  call it a strict representation embedding.
	\item We write $A \leq B$ resp. $A \leq_{s} B$  if there is a representation embedding  resp. a strict representation embedding from $mod\,A$ to $mod\,B$.  
 
\item Two algebras $A$ and $B$ with $A \leq B$  and $B \leq A$  are called embedding equivalent. They belong to the same embedding class. Analogous definitions 
hold for strict representation embeddings.
	\item An algebra $A$ with $k\langle X,Y \rangle \leq A$ resp. $k\langle X,Y \rangle \leq_{s} A$  is called  wild resp. strictly wild.

	\end{enumerate}
	\end{definition}

The three conditions a representation embedding has to satisfy are independent of each other as trivial examples of functors having  only two of the required properties show. 
The following  observations are often useful. 
\begin{lemma}
 Let $F:mod \,A \longrightarrow mod \,B $ be an exact $k$-linear functor that preserves indecomposability. Then we have:
\begin{enumerate}
 \item $F$ is a representation embedding if and only if it reflects the isomorphism classes of indecomposables.
\item $F$ induces injections on the spaces of homomorphisms and extensions.
\end{enumerate}

\end{lemma}

Analogous definitions and statements hold for representations of quivers.
 Sometimes we consider a functor $F:C \longrightarrow D$  where $C$ and $D$ are  full subcategories of  module categories 
which are closed under extensions, images and isomorphisms. Then $C$ and $D$ inherit the structure of an exact category and  they are closed under direct summands. 
The definitions and  lemma 1 carry over to this more general situation.

 \subsection{First results and examples}
	
The following result applies in particular to representation embeddings.
	
\begin{proposition} Let $F: mod\,A \longrightarrow mod \,B$ be a faithful, $k$-linear, exact functor where $A$ and $B$ are finite-dimensional algebras. Then we have:
\begin{enumerate}
 \item $F(A)=_{B}M_{A}$ is a bimodule where the action of $A$ is induced via $F$ from the isomorphism $End (_{A}A) \simeq A^{op}$.
\item $F$ is isomorphic to the functor $X \mapsto M\otimes X $ where the tensor product is taken over $A$.
\item $M_{A}$ is a progenerator.
\item Set $A_{A}=\oplus_{i=1}^{r} P_{i}^{n_{i}}$ with pairwise non-isomorphic indecomposable $P_{i}$'s and suppose that  $M_{A}=\oplus_{i=1}^{r} P_{i}^{m_{i}}$.
Then we have $m_{i}= dim FS_{i}$. Thus $M_{A}$ is free of rank $q$ iff $dim FS_{i}=q dim S_{i}$ holds for all simples. 
\item $F$ is the composition of a Morita-equivalence and the restriction corresponding to the ring homomorphism $\phi: B \longrightarrow End (M_{A})$ giving the bimodule-structure.
\end{enumerate}
 
\end{proposition}

\begin{proof}
 That $F(A)$ is a bimodule is trivial and that $F$ can be described by the tensorproduct is a nice observation of Eilenberg and Watts (   \cite[page 312]{SS} ) that holds for
 any right-exact functor.

 The exactness means that 
$M$ is flat as an $A$-module. But the adjoint isomorphism $Hom _{k}(M\otimes X, k)\simeq Hom_{A}(X,Hom_{k}(M,k))$ shows that 'flat' is the same as ' projective' in $ mod\,A^{op}$. 
Because $F$ is faithful $M$ is a generator. The rest is easy. 
\end{proof}

The case where $F: mod \,A \longrightarrow mod \,B$ is given by tensoring with a bimodule which is free over $A$ occurs often when $A$ is the free algebra in several unknowns and $B$ is 
finite dimensional. Then we call
the bimodule affine if there is an affine basis i.e. an $A$-basis $m_{1},m_{2},\ldots ,m_{n}$ such that for all $b \in B$ and all $i$ the vector $bm_{i}$ is an $A$-linear combination of 
the basis with coefficients ( depending on b )
 of degree $\leq 1$. Observe that the rank  of a free $A$ module is uniquely determined because the groundfield $k$ is a residue algebra of $A$. There is the following simple observation:

\begin{lemma} Let $F: mod\,A \longrightarrow mod \,B$ and  $G:mod \,B \longrightarrow mod \,C$ be  exact functors which are given by tensoring with 
bimodules $_{B}M_{A}$ and $_{C}N_{B}$ 
that are finitely  generated as right modules. Suppose $B$ and $C$ are finite-dimensional.
Then we have:
\begin{enumerate}
 
 \item $G\circ F$ is induced by tensoring with $_{C}L_{A}=_{C}N \otimes M_{A}$.
\item If $_{B}M_{A}$ is free over $A$ of rank $m$ and $_{C}N_{B}$ is free over $B$ of rank $n$ then $_{C}L_{A}$ is free over $A$ of rank $mn$.
\item Let $A$ be $k\langle X_{1},X_{2},\ldots ,X_{n} \rangle$. If $_{B}M_{A}$ is affine and $_{C}N_{B}$ is free then $_{C}L_{A}$ is affine. 
\item Let $A$ be $k\langle X_{1} \rangle$. If $_{B}M_{A}$ is free over $A$  then $_{C}L_{A}$ is free.
\end{enumerate}

\end{lemma}
\begin{proof}
 The proof is straightforward. For part three one verifies that the tensors $n_{i}\otimes m_{j}$ formed by the members of an arbitrary basis of $N$ and an affine basis of $M$ 
form an affine basis of the tensor product. The last part  uses that $_{C}L_{A}$ is a direct summand whence a submodule of a free module. A standard result of basic algebra 
says that
 $_{C}L_{A}$ is
 free. In fact by a result of  Cohn  this holds also for $n>1$.
\end{proof}

We end this section with a useful strict representation embedding which  shows how to get rid of composable arrows.
Let $Q$ be an arbitrary quiver. We define a new quiver $\tilde{Q}$ by dividing each point $p$ of $Q$ into an emitter $p^{+}$ and a receiver $p^{-}$. An arrow $\alpha$ in $Q$ from $p$ to $q$
induces in $\tilde{Q}$ an arrow from $p^{+}$  to $q^{-}$ again called $\alpha$.  Besides the induced arrows one has in $\tilde{Q}$ for each point $p$ of $Q$ an additional arrow $\alpha_{p}$ 
from $p^{+}$ to $p^{-}$. It is clear that $\tilde{Q}$ contains only sources and sinks.
Given a representation $M$ of $Q$ we define the representation $FM$ of $\tilde{Q}$ by $(FM)(p^{+})=(FM)(p^{-})=M(p)$, by $(FM)(\alpha)=M(\alpha)$ for the induced arrows and 
by $(FM)(\alpha_{p}) = id_{M(p)}$
for the additional arrows.

Obviously $F$ induces an equivalence with the full exact subcategory of $rep \,\tilde{Q}$ where all additional arrows are represented by bijections. As a special case  one gets for 
 $L_{n}$ a strict representation embedding
 $F: rep\, L_{n} \longrightarrow rep\, K_{n+1}$. The reader can verify that it is given by tensoring with an affine bimodule of rank $2$.

\section{Some  examples of representation embeddings}

\subsection{Two results of Jans resp. Gabriel and of  Gelfand and Ponomarev via representation embeddings}

In 1957 Jans formulated the famous Brauer-Thrall conjectures in \cite{Jans} and he proved the second conjecture for certain algebras. Generalizing results of Brauer, Nakayama and Thrall 
he attached a bipartite quiver $K=K(A,I)$
to an algebra $A$ and a twosided ideal $I$ lying in the two-sided socle and the radical. If this quiver contains a double-arrow, an $\tilde{A}_{n}$ or a $\tilde{D}_{n}$ he constructed
 case by case  in infinitely many dimensions infinitely many pairwise non-isomorphic indecomposables. At that time the notions of a quiver and its representations were not yet properly 
formulated and so the proofs are lengthy with many verifications left to the reader. 

In section 9 of the survey article \cite{GI2} Gabriel interpreted the results of Jans by defining a functor $G:mod\,A \longrightarrow rep\,K$ that induces a bijection 
between the isomorphism
classes in a certain full subcategory $C$ of $mod\,A$ and the full subcategory $rep_{'}K$ of all representations not containing a simple projective as a direct summand. Here $C$ consists of
all modules $X$ such that $X/TX$ is a projective $A/T$-module. Now the functor  $G$ is not exact and it is not obvious for which exact structure of $C$ the restriction of $G$ to $C$
 becomes exact.

Here we define a representation embedding $F:rep'K \longrightarrow  mod\,A$ where $K$ is now the  quiver opposite to the quiver constructed by Jans and where $rep'K$ is 
the full subcategory of representations that have no simple injective direct summand. This subcategory is closed under extensions, submodules  and isomorphisms whence the term 'representation embedding'
makes sense.

So we consider a finite-dimensional algebra $A=kQ/I$ and a twosided ideal $T$ contained in the radical and annihilated on both sides by the radical. Any point $x$ in $Q$ splits into two 
points $x^{+}$ and $x^{-}$ of the bipartite quiver $K$ and  we choose a base vector $v_{x}$ in each simple $A$-module $S_{x}$. 
For any two points $x,y$ in $Q$ we set
$n_{xy}= \,dim \,e_{x}Te_{y}$, we draw $n_{xy}$ arrows from $x^{+}$ to $y^{-}$ and we denote them by $\alpha(x,y,j)$ ,$1 \leq j \leq n_{xy}$. We also choose a basis 
$s(x,y,j)$, $1 \leq j \leq n_{xy}$ of $e_{x}Te_{y}$ for each pair with $n_{xy}\neq 0$. Of course $K$ need not be connected even if $Q$ is so.

Now we construct a functor $F: rep'K \longrightarrow mod\,A$. Given a representation $M$ in $rep'K$ we look at the map 
$$f_{M}: S_{M}=\oplus _{x \in Q} S_{x}\otimes M(x^{+}) \longrightarrow P_{M}:=\oplus_{y\in Q} Ae_{y} \otimes M(y^{-})$$ that sends $v_{x}\otimes v \in S_{x}\otimes M(x^{+})$ to $\sum_{y,j}\, s(x,y,j) \otimes \alpha(x,y,j)(v)$.
 
The tensor product is taken over $k$. Note that $M(x^{+})=0$   if there is no arrow starting in $x^{+}$ because $M$ is in $rep'K$.  Here $S_{M}$ is semi-simple, $P_{M}$ is projective and $f_{M}$ is $A$-linear 
with image  contained in the radical. Thus the projection
 $p_{M}: P_{M} \longrightarrow FM:=Coker \,f_{M}$ is a projective cover. Moreover $f_{M}$ is injective because $M$ belongs to $rep'K$. 

Any morphism $g:M  \longrightarrow M'$ induces in
an obvious way 
$A$-module homomorphisms
$g_{1}=\oplus_{x} \, (id_{S_{x}}\otimes g(x^{+})):S_{M} \longrightarrow S_{M'}$ and $g_{2}=\oplus_{y} \, (id_{Ae_{y}}\otimes g(y^{-})):P_{M} \longrightarrow P_{M'}$ such that $f_{M'}\circ g_{1}= g_{2}\circ f_{M}$. Thus we get an induced 
homomorphism $Fg: FM \longrightarrow FM'$ and the wanted functor is defined.

\begin{theorem}  We keep all the assumptions and notations introduced above. Then the  following holds:
\begin{enumerate}
 \item $F$ is a representation embedding.
\item The essential image  of $F$ is the category $C$ defined before.
\item  $F$ is strict iff  $Q$ contains no arrows.
\item If $ n:=n_{xy} \geq 2$ for some not necessarily distinct points $x,y$ then there is a representation embedding 
$$H: mod \,k\langle X_{1},X_{2},\ldots ,X_{n-1}\rangle \longrightarrow mod\,A,$$
that is induced by an affine bimodule of rank $dim\,Ae_{y}\,-\,1$.
\end{enumerate}

\end{theorem}
\begin{proof} Of course, $F$ is $k$-linear. It is exact by the snake lemma.

Let $\rho: FM \longrightarrow FM'$ be a homomorphism. It can be lifted to 
a homomorphism $\sigma:P_{M} \longrightarrow P_{M'}$ between the projective covers and we also get an induced homomorphism $\tau:S_{M} \longrightarrow S_{M'}$ between the kernels. 
Now $Hom_{A}(Ae_{y},Ae_{y'})$ consists in right multiplications by elements of $e_{y}Ae_{y'}$. For $y\neq y'$ all these elements lie in the radical of $A$ and for $y=y'$ we can choose a basis
consisting of $e_{y}$ and elements in the radical. Therefore we obtain $ \sigma = \oplus_{y \in Q} \,(id_{Ae_{y}}\otimes h(y^{-}))  + \sigma '$, where 
$h(y^{-}):M(y^{-}) \longrightarrow M'(y^{-})$ is a 
uniquely determined linear map and $\sigma '$ is a bloc matrix of tensors $r \otimes s$ such that $r$ annihilates all $s(x,y,j)$. All these blocs are in the radical of the category $mod\,A$.
  Since $S_{M}$ and $S_{M'}$ are semisimple we have $\tau= \oplus_{x \,\in Q} \,id_{S_{x}}\otimes h(x^{+})$ for uniquely determined linear maps $h(x^{+}):M(x^{+}) \longrightarrow M'(x^{+})$.
 Calculating the image of
 any $v_{x}\otimes v \in S_{x}\otimes M(x^{+})$ under $\sigma \circ f_{M}=f_{M'} \circ \tau$ in two ways we see that the family of the $h(x^{+})$ and $h(y^{-})$ defines a morphism $h:M \longrightarrow M'$. 

If $\rho$ is an isomorphism then the lifting $\sigma$  is an isomorphism because $p_{M}$ and $p_{M'}$ are projective covers, whence $\tau$ is an isomorphism. From standard properties
 of the radical of $mod\,A$  it follows 
that $h$ is an isomorphism.

Similarly for an indecomposable $M$ any idempotent $\rho:FM \longrightarrow FM$ gives rise to an idempotent $h:M \longrightarrow M$ such that $\rho -Fh$ lies in the radical of $mod\,A$.
Thus for $h=0$ the idempotent $\rho$ lies in the radical whence $\rho=0$ and for $h=id_{M}$ the idempotent $\rho$ is invertible whence $\rho=id_{M}$.

Since the image of $f_{M}$ is contained in $TP_{M}$ we have $P_{M}/TP_{M} \simeq FM/TFM$ and so $FM$ lies in $C$. Reversely for any such module
 $X$ one looks at an exact sequence $0 \longrightarrow S \longrightarrow P \longrightarrow X \longrightarrow 0$ where $p:P \longrightarrow X$ is a projective cover. Then $i:S \longrightarrow P$
is given by an $f_{M}$ for some appropriate $M$.

Next, the simples  $S_{y^{-}}$ are mapped by $F$ to the projectives $Ae_{y}$ for all $y \in \,Q$. Thus these modules are pairwise orthogonal bricks if $F$ is strict and there are no 
arrows in $Q$. The reverse direction is even more trivial.

Finally we set $B:= k\langle X_{1},X_{2},\ldots ,X_{n-1}\rangle$. The full subquiver of $K$ supported by $x^{+}$ and $y^{-}$ is isomorphic to the Kronecker-quiver $K_{n}$ and so we obtain a 
representation embedding $H$ from $mod \,B$ to $mod\,A$ by composing the obvious embeddings $mod\,B \longrightarrow rep'K_{n}$ and $rep'K_{n} \longrightarrow rep'K$ with $F$.
It only remains to verify that $H$ is induced by tensoring with an affine bimodule $_{A}M_{B}$. So we consider the bimodule 
homomorphism
$f: S_{x}\otimes B  \longrightarrow  Ae_{y} \otimes B$ that sends $v_{x}\otimes 1$ to $s(x,y,n)\otimes 1  \,+ \, \sum_{j}^{n-1}\, s(x,y,j) \otimes X_{j} $. 
The cokernel is the wanted affine bimodule $_{A}M_{B}$.
Namely let $p_{1}=s(x,y,1),p_{2}=s(x,y,2),  \ldots, p_{n}=s(x,y,n),p_{n+1}, p_{m}$ be a $k$-basis of $Ae_{y}$. Then $_{A}M_{B}$ has the residue classes 
of the  $p_{i} \otimes 1$ with $i \neq n$ as a $B$-basis and $H$ is given by tensoring with $_{A}M_{B}$. If we have $ap_{i}= \sum_{j=1}^{m}\rho_{ji}p_{j}$ with $a$ in $A$ and $\rho_{ji}$ 
in $k$ then one has $a(p_{i}\otimes 1)= \sum_{j=1}^{m}\rho_{ji}(p_{j}\otimes 1)$ in $Ae_{y}\otimes B$ and in the cokernel one replaces $p_{n}\otimes 1$ by $-\sum_{j=1}^{n-1} p_{j}\otimes X_{j}$.
Thus the bimodule has an affine basis.

\end{proof}

We just have seen that representation embeddings occur implicitely already in the work of Jans; explicitely they appear 1969 in the short article \cite{GP} of Gelfand and Ponomarev
where they prove that the commutative algebra $k[X,Y]/(X,Y)^{3}$ is wild so that - somewhat surprisingly - the classification of pairs of commuting nilpotent matrices up 
to simultaneous similarity
is as complicated as the same problem for pairs of arbitrary matrices. Gelfand and Ponomarev describe an embedding of $mod\,k\langle X,Y \rangle$ into
$mod \,k[X,Y]/(X,Y)^{3}$ in the language of matrices. Their result is also a simple consequence of theorem 1.
\begin{corollary}
 For each $n\geq 2$ there is a representation embedding $$mod \,k\langle X_{1},X_{2}, \ldots ,X_{n}\rangle \longrightarrow mod\,k[X,Y]/(X,Y)^{n+1}$$ that is given by tensoring with an
affine bimodule.
\end{corollary}
\begin{proof}
 One applies part iv) of theorem 1 to the algebra $k[X,Y]/(X,Y)^{n+1}$ and the ideal $(X,Y)^{n}/(X,Y)^{n+1}$.
\end{proof}

Recall that a finite dimensional algebra $A$ is called distributive iff its two-sided ideals form a distributive lattice. Jans has shown in \cite{Jans} that this is true iff the lattice is 
finite and Kupisch has
 proven in \cite{Kup} that it is equivalent to the fact that for all primitive idempotents $e,f$ the rings $eAe$ are uniserial and the $eAf$ are uniserial as $eAe-fAf$-bimodules. In fact $eAf$
is then already uniserial from the right or from the left. These observations are important but easy to show. 

Part of the next corollary is already contained in the article of Jans.

\begin{corollary}
 If $A$ is basic and not distributive there is a representation embedding $rep'K_{2} \longrightarrow mod\,A$ whose composition with the embedding $mod\,k[T] \longrightarrow rep'K_{2}$
is given by the tensor-product with an affine bimodule.
\end{corollary}
\begin{proof}
 There are two points $x,y$ in the quiver of $A$ such that $e_{x}Ae_{y}$ is not uniserial. Thus there is an index $i$ such that $ dim \,R^{i}/R^{i+1} \geq 2$. Here $R^{j}$ denotes the $j^{th}$ 
radical of the bimodule $e_{x}Ae_{y}$. Now let $T$ be the two-sided ideal generated by $R^{i}$ and let $J$ be $(rad A)R^{i}+R^{i}(rad A)$.
In the residue algebra $B=A/J$ the ideal
$T/J$ satisfies $n_{xy}\geq 2$. Thus we get the wanted representation embedding into $mod\, B$ that we can compose with the obvious embedding into $mod\,A$ which is given by tensoring
 with $B$. The result follows from lemma 2.
\end{proof}

As another consequence of theorem 1 we give a simple proof of the following well-known fact ( see \cite{SS,GTW} ).

\begin{proposition}
 If a finite-dimensional algebra $A$ is wild there is a representation embedding which is given by tensoring with an affine bimodule $_{A}M_{k<X,Y>}$.
\end{proposition}
\begin{proof}
 Set $B=k<X,Y>$ and $C=k[T_{1},T_{2}]/(T_{1},T_{2})^{3}$. Since $B$ projects to $C$ and by assumption we have representation embeddings $mod \,C \longrightarrow mod\, B$ and
 $mod\,B \longrightarrow mod \,A$. By proposition 1 the composition $F$ is a representation embedding given by tensoring with a bimodule $_{A}M_{C}$ which is finitely 
generated projective  over the local algebra $C$ whence free because $C$ is the  only  indecomposable projective. By corollary 1 we have a representation embedding
 $H:mod  \,B \longrightarrow mod\,C$
which is given by tensoring with an affine 
bimodule $_{C}N_{B}$. Then  lemma 2 shows that $F \circ H$ is induced by an  affine bimodule.

\end{proof}

\subsection{A refinement of Brenners embedding}

In this section we  symmetrize a construction  due to Brenner  ( see  \cite{Brenner} ) which gives a strict representation embedding 
$F: mod\, A \longrightarrow mod \,k\langle X,Y \rangle$ for any finite finitely generated algebra $A$. 
The usual proof  consists in easy but lengthy matrix calculations ( see e.g. \cite[page 316]{SS}     ). Here we minimize these matrix calculations and we use  the simple 
 underlying shifting-structure 
which occurs already in the irreducible complex finite-dimensional representations of $sl_{2}({ \bf C})$.
We find as an additional property of $F$ that it induces an isomorphism between the lattices of submodules and therefore it preserves the length.

So let $n\geq 2$ be a natural number and let $A$ be the free associative algebra in $n(n+1)$ variables. To make the notations easier we denote these by $$X_{31},X_{41},X_{42},\ldots,
X_{n+2,1},X_{n+2,2},\ldots ,X_{n+2,n}$$ and by $$Y_{13},Y_{14},\ldots ,Y _{1,n+2},Y_{24},\ldots,Y_{2,n+2},\ldots , Y_{n,n+2}$$ with transposed indices. An $A$-module is nothing but
 a vector space $M$ endowed with endomorphisms $x_{ij},y_{ij}$ given by the action of the variables. The $k\langle X,Y \rangle$-module FM is defied as the vector space $M^{n+2}$. The action of $x$ is 
given by a strict lower triangular $(n+2)\times (n+2)$ bloc-matrix whose non-zero blocs are the identity at the positions $(2,1),(3,2),\ldots ,(n+2,n+1)$ and  $x_{ij}$ at the positions
 $(i,j)$. Symmetrically the action of $Y$ is given by the 'transposed' upper triangular bloc.matrix with the $y_{ij}$ as entries at the obvious positions. So for $n=2$ the shape of $x$ and 
$y$ is

 \[
\left [
\begin{array}{cccc}
0&0&0&0\\
1&0&0 & 0\\
x_{31}&1&0&0\\
x_{41}&x_{42}&1&0
\end{array}
\right ] and 
\left [
\begin{array}{cccc}
0&1&y_{13}&y_{14}\\
0&0&1 & y_{24}\\
0&0&0&1\\
0&0&0&0
\end{array}
\right ].
\]

The definition of $F$ on morphisms is the obvious diagonal one.

\begin{theorem} Let $F: mod\, A \longrightarrow mod \,k\langle X,Y \rangle$ be defined as above. Then we have:
\begin{enumerate}
 \item $F$ is $k$-linear, exact and fully faithful.
\item Any submodule of $FM$ is of the form $FM'$ for some submodule $M'$ of $M$. Thus $F$ induces an isomorphism between the lattices of submodules and  it preserves the length.
\end{enumerate}
\end{theorem}
\begin{proof}
 The elements of $M^{n+2}$ are columns with entries in $M$. However, to save space, we write these as rows. In the same vein a morphism from $FM$ to $FN$ is written as a bloc matrix 
with entries in $Hom_{k}(M,N)$. The point is that $x$ is just a shift operator given by $$(0,0, \ldots ,0,m,\star,\ldots ,\star) \mapsto  (0,0, \ldots ,0,0,m,\star,\ldots ,\star).$$
Here $m$ is an arbitrary element in $M$ at the position $i$ in the first vector and the position $i+1$ in the second. The stars in the first row are arbitrary elements in $M$ and the 
stars in the second row are obtained from these by matrix multiplication involving the $x_{ij}$ which is lengthy to write down and unimportant for most of the arguments. The situation is
symmetric for $y$.

In part i) only the surjectivity of $F$ on the homomorphism spaces is not obvious. However using the shift-operation several times one sees easily for $i=n+1,n,n-1, \ldots,1$ that 
$x^{i}FM= 0^{i}\times M^{n+2-i}$ consists of all columns having $0$ at the first $i$ positions. Of course $FN$ has a similar filtration obtained by the action of powers of $x$. Any 
homomorphism $\alpha:FM \longrightarrow FN$ has to respect these filtrations i.e. $\alpha$ is a lower triangular bloc matrix. By symmetry, $\alpha$ is also an upper triangular bloc matrix 
whence
 a diagonal bloc-matrix. Again by looking at the shift operation of $x$ one gets that all diagonal entries are equal to the same linear map $\beta:M \longrightarrow N$. Then $x\alpha=\alpha x$ resp. $y\alpha =\alpha y$ just means that $\beta$ commutes with all $x_{ij}$ resp. $y_{ij}$, i.e. $\beta$ is the wanted homomorphism with 
$F\beta=\alpha$.

In part ii) let $U$ be a submodule of $FM$ and let $\pi_{i}:M^{n+2} \longrightarrow M$ be the projection onto the $i$-th component. We define $M'= \pi_{1}U$.
First we use both shift operations to prove by induction $M'= \pi_{i}U$ and $0^{i-1}\times M' \times 0^{n+2-i} \subseteq U$. For $i=1$ we have $y^{n+1}x^{n+1}U = M' \times 0^{n+1}$. 
This belongs to $U$ because $U$ is a submodule. In the induction step we set $M''=\pi_{i+1}(U)$. The inclusion $M'\subseteq M''$ follows by applying $x$ once to
 $0^{i-1}\times M' \times 0^{n+2-i} \subseteq U$. For any $u=(m_{1},m_{2},\ldots, m_{i},m_{i+1},\ldots )$ also $u'=(0,0,\ldots,0,m_{i+1}, \ldots )$ belongs to $U$ by the inductive hypothesis.
Therefore $x^{n+1-i}u'= (0,0,\ldots,0,m_{i+1})$, $y^{n+1}x^{n+1-i}u'=(m_{i+1},0,0,\ldots, 0)$ and $y^{n+1-i}x^{n+1-i}u'=(\star,\star,\ldots,m_{i+1},0,\ldots,0)$ belong to $U$. Thus
 $m_{i+1}\in M'$ and $0^{i}\times M' \times 0^{n+1-i}\subseteq U$ follow.

We have shown that $U=M'^{n+2}$ only by using the shift-operations. To see that $M'$ is stable under all $x_{ij}$ one writes out for arbitrary $i$ the condition
 $x(0^{i}\times M' \times 0^{n+1-i}) \subseteq M'^{n+2}$. Symmetrically $M'$ is stable under all $y_{ij}$ i.e. a submodule. Finally we have $FM'=U$. Therefore $F$ induces a bijection
 between the lattices of submodules which is even an isomorphism.

\end{proof}

The functor induces morphisms at the geometric level between  the corresponding varieties  of 
representations, but these do not have good properties. This drawback is removed  in \cite[page17]{BoGeo} by yet another variant of the 
above construction which shows for instance that any singularity occurring in an orbit closure
 of a module over any algebra occurs already for a $k\langle X,Y \rangle$-module.

\subsection{Representation embeddings via extensions}

 Let $A$ be an arbitrary associative algebra and let $U,V$ be two indecomposable modules of finite dimension. Recall that $Ext_{A}(V,U)$ can be described as the quotient of the space of 
all derivations 
$$Z(V,U)= \{z:A \rightarrow Hom_{k}(V,U) \mid \,z(ab)= u(a)z(b)+z(a)v(b)\}$$ by the space 
$$B(V,U)=\{z \in Z(V,U) \mid \,z(a)= u(a)h-hv(a) \mbox{for some }h \in Hom_{k}(V,U) \}$$
 of all inner derivations. Suppose we have $n$ derivations $z_{1},z_{2},\ldots ,z_{n}$ whose residue classes are linearly independent in $Ext_{A}(V,U)/J$ where $J$ is the radical of the  
$End_{A}(U)-End_{A}(V)$-bimodule $Ext_{A}(V,U)$. In the following  all tensor products are taken over $k$. Furthermore the natural isomorphism 
$Hom_{A}(V\otimes X,U \otimes Y) \simeq Hom_{A}(V,U) \otimes Hom_{k}(X,Y)$  carries over to the bifunctors $Z,B,Ext$ and we view these as identifications.

Now we define a  functor $F: rep \,K_{n}  \rightarrow mod \,A$ by sending a representation $W=(W_{1},W_{2};\phi_{1},\phi_{2},\ldots \phi_{n})$ to
the vector space $FW=(U\otimes W_{2})\oplus (V\otimes W_{1})$.  Writing elements of $End_{k}FW$  as 
$2\times2$-bloc matrices with entries in $Hom_{k}(U,U) \otimes Hom_{k}(W_{2},W_{2})$ and so on  
the $A$-module structure on $FW$ 
is given by 
$$ a \mapsto 
\left [
\begin{array}{cc} 
u(a)\otimes id_{W_{2}}&\sum_{i=1}^{n} \,z_{i}(a)\otimes \phi_{i}\\
0&v(a) \otimes id_{W_{1}}
\end{array}
\right ]. 
$$

 One checks that this map is an algebra-homomorphism and one defines $F$ on morphisms in the obvious way by a $2\times 2$-diagonal matrix to obtain a $k$-linear exact functor. 
From now on we suppress the variable $a$ in our calculations.

\begin{theorem}
 We keep all  assumptions and notations and we assume in addition that $Hom_{A}(U,V)=0$. Then we have:
\begin{enumerate}
 \item  $F$ is a representation embedding.
\item $F$ is full iff $U$ and $V$ are orthogonal bricks.
\item For $n\geq 2$ the composition of $F$ with the representation embedding from section  2.1  is induced by an affine bimodule of rank $ dim (U\oplus V)$.
\end{enumerate}

\end{theorem}

\begin{proof} This is just a lengthy verification using matrix calculations. Nevertheless we give full details.
 We write a homomorphism $\alpha$ from $FW$ to $FW'$ as a $2\times 2$ bloc-matrix $
\left [
\begin{array}{cc}
P&Q\\
R&S
\end{array}
\right ]. 
$ Then we have the matrix equation  $$\left [
\begin{array}{cc}
P&Q\\
R&S
\end{array}
\right ]\left [
\begin{array}{cc}
u\otimes id_{W_{2}}&\sum_{i=1}^{n} \,z_{i}\otimes \phi_{i}\\
0&v \otimes id_{W_{1}}
\end{array}
\right ]=$$ $$
\left [
\begin{array}{cc}
u\otimes id_{W'_{2}}&\sum_{i=1}^{n} \,z_{i}\otimes \phi'_{i}\\
0&v \otimes id_{W'_{1}}
\end{array}
\right ]
\left [
\begin{array}{cc}
P&Q\\
R&S
\end{array}
\right ].
$$ 

It follows that $R \in Hom_{A}(U,V)\otimes Hom_{k}(W_{2}, W'_{1})$ whence $R=0$. Next we find $P \in End_{A}(U) \otimes Hom_{k}(W_{2},W'_{2})$.
 We take a basis $\pi_{1},
\pi_{2},\ldots \pi_{p}$ of the radical of $End_{A}U$ and we set $\pi_{0}=id_{U}$. Then we get $P=\sum_{i=0}^{p} \,\pi_{i}\otimes P_{i}$ with uniquely determined 
$P_{i} \in Hom_{k}(W_{2},W'_{2})$.
The case $p=0$ can occur here.

 Similarly we write $S =\sum_{i=0}^{s} \,\sigma_{i}\otimes S_{i}$ with uniquely determined maps
$S_{i}$ in $Hom_{k}(W_{1},W'_{1})$.  Finally we choose any basis $Q_{ij}$ of $Hom_{k}(W_{1},W'_{2})$ and we write $Q=\sum _{i,j} Q'_{ij}\otimes Q_{ij}$ with uniquely determined
 $Q'_{ij} \in Hom_{k}(V,U)$. We compare the entries in the right upper corner of the matrix equation and we find in $Z(V,U)\otimes Hom_{k}(W_{1},W'_{2})$ 
the equation:
$$0=\sum_{i=1}^{n} z_{i}\otimes (P_{0}\phi_{i}-\phi'_{i}S_{0}) \,+ \sum _{i,j}(Q'_{ij}v-uQ'_{ij})\otimes Q_{ij} +$$
 $$\sum_{i=1}^{p} \sum _{j=1}^{n}\pi_{i}z_{j}\otimes P_{i}\phi_{j}  - \sum_{i=1}^{s} \sum _{j=1}^{n}z_{j}\sigma_{i}\otimes \phi'_{j}S_{j}.$$

Projecting this to $(Ext(V,U)/J)\otimes Hom(W_{1},W'_{2})$ we obtain that $(P_{0},S_{0})$ belongs to $Hom_{kK_{n}}(W,W')$. 

Suppose now that $\alpha$ has an inverse $\alpha'$. Then we find as above a homomorphism $(P'_{0},S'_{0})$ which is the inverse of $(P_{0},S_{0})$.
Next assume $W$ is indecomposable and that  $\alpha$ is an idempotent in $End_{A}FW$. Then $(P_{0},S_{0})$ is an idempotent in $End_{kK_{n}}(W)$ whence $0$ or the identity.
Now $\alpha - F(P_{0},S_{0})$ is nilpotent and so $\alpha$ is the identity or zero.

If $F$ is in addition full then $U$ and $V$ are orthogonal bricks because they are the images of the simples under $F$. Conversely if $U$ and $V$ are orthogonal bricks  we have $p=s=0$ and
$Q'_{ij} \in Hom_{A}(V,U)=0$ for all $i,j$. It follows that $F$ is full.

Abbreviate $k\langle X_{1},X_{2},\ldots X_{n-1}\rangle $ by $B$. The $A$-module structure on the bimodule $_{A}M_{B}= (U\otimes B) \oplus (V\otimes B)$ is given by the map 

$$ a \mapsto 
\left [
\begin{array}{cc} 
u(a)\otimes id_{W_{2}}& \sum_{i=1}^{n-1} \,z_{i}(a)\otimes X_{i} + z_{n}(a)\otimes 1\\
0&v(a) \otimes id_{W_{1}}
\end{array}
\right ]. $$

\end{proof}
The case of two orthogonal bricks was of course known before and only this case is used in the next proof. Later on in section 6.4 we need the general result.
\begin{proposition}
 Let $R$ be a quiver with two orthogonal bricks $U,V$ having the same dimension vector $\alpha$. Suppose $m:=-q_{R}(\alpha)$ is $\geq 1$. Then there is a strict representation embedding
$F: rep\,K_{m} \longrightarrow rep\,R$ such that the dimension vectors of all $FX$ are integer multiples of $\alpha$.
\end{proposition}

\begin{proof}The well-known equality $ q_{R}(\alpha) = dim\, Hom (V,U) - dim Ext(V,U)$ shows $dim\,Ext(V,U)=m$. By the last theorem the wanted $F$ exists.
 
\end{proof}

 Peternell, one of my students, has elaborated the idea of theorem 3 in a much more general setting   in  \cite{P} and Weist has adapted this to his 
purposes in \cite{W}. 

\section{Simultaneous orthogonal strict embeddings into the representation category of a wild quiver}

\subsection{The main result}
 For any wild quiver $Q$ and any finitely generated algebra $A$ there is a strict representation embedding $mod \,A \longrightarrow rep \,Q$ as is well-known. 
To prove this it is by Brenners theorem sufficient  to construct explicitely in terms of matrices  a strict embedding of $rep L_{2}$ into $rep \,Q$ for the easily 
determined finite list of all 
minimal wild quivers $Q$.

That way the 
the result was found and proven for the first time ( see \cite{Brenner,GRI} ).       
In the book   \cite{SS}  there is another proof using Brenners embedding and an inductive argument based on perpendicular categories.
Here we present a souped-up version of the theorem that requires only very few explicit matrix 
calculations and in particular not Brenners embedding.

\begin{theorem} Let $Q$ be a wild quiver and let $(A_{i})_{i \in I}$ be a family of finitely generated algebras where the index set $I$ has at most the cardinality of $k$. 
Then there are strict representation embeddings
 $$ F_{i}: mod\,A_{i} \longrightarrow rep\,Q$$ with the following properties:
\begin{enumerate}
 \item The embeddings are  pairwise orthogonal i.e.  $Hom(F_{i}(X),,F_{j}(Y))=0$ for $i\neq j$ and arbitrary $X \in mod\,A_{i}$, $Y \in mod \,A_{j}$.
\item There is a root $\alpha$ of $Q$ with $1 \leq p:= -q_{Q}(\alpha)$ such that the dimension vectors of all $F_{i}(X)$ are multiples of $\alpha$ and such that for all $n$ the unions of 
the orbits of all $F_{i}X$ with dimension-vector $n\alpha$ lie in a proper closed subset of the variety $rep(Q,n\alpha)$.

\end{enumerate}
\end{theorem}

Note that we can take for the family $(A_{i})$ a list of the isomorphism classes of  all finite-dimensional $k$-algebras and so all the corresponding module-categories are fully embedded
 simultaneously 
and pairwise $hom$-orthogonal.
Furthermore part ii) says that all these embeddings require only very few space inside the whole module category of the wild quiver. They all live in the representation varieties 
 to very few  dimension-vectors and 
then they lie in proper closed subvarieties.

\subsection{Two matrix calculations}

\begin{lemma}
 $L_{2}$ has in each dimension $n$ a family of pairwise non-isomorphic simples $S(\lambda )$, where  $\lambda$ varies in a cofinite subset $I(n)$ of $k$.
\end{lemma}
\begin{proof}
 We just give the action of the two variables $X$ and $Y$ on $k^{n}$. First we fix a finite set $M$ of $n-1$ elements in  $k$ and we take $I(n)$ as its complement. Then for
 $\lambda \in I(n)$ we take as $x$ the diagonal matrix with $\lambda$ and the elements of $M$ as diagonal entries and for $y$ the  matrix corresponding to the cyclic permutation
$1 \mapsto 2 \mapsto 3 \ldots n-1 \mapsto n \mapsto 1$. Any non-zero submodule $U$ contains an eigenspace of $x$ whence all eigenspaces by the $y$-action. Thus $S(\lambda)$ is always simple
 and two simples to different indices are not isomorphic because $\lambda$ is only in $S(\lambda)$ an eigenvalue of $x$. 

\end{proof}

\begin{lemma}
 Let $Q$ be a finite connected wild quiver. Then there is a root $\alpha$ with 
$q_{Q}(\alpha) \leq -3$ such that  three orthogonal bricks $U,V, W$ exist with dimension vector $\alpha$.
\end{lemma}
\begin{proof}
Assume first that $Q$ contains a cycle $C$, oriented or not. Choose such a cycle of minimal cardinality $l$ and denote its vertices by $x_{1},x_{2}, \ldots ,x_{l}$ and its arrows by $\alpha_{1},
\alpha_{2},\ldots, \alpha_{l}$. For a  subset $M$ of $k$ with $4$ elements we define a representation $S=S(M)$ of $C$ by $S(x)=k^{4}$ for all $x \in C$. The arrows are all represented
by the unit-matrix except for $S(\alpha_{1})$ which is a diagonal matrix with $M$ as the set of diagonal elements. This representation of $C$ is the direct sum of pairwise orthogonal 
one-dimensional bricks. Because $Q$ is wild and connected there is at least one more arrow $\beta:y \rightarrow z$  in $Q$ with $y$ or $z$ in $C$. 

If $y$ and $z$ both belong to $C$ we take for
 $S(\beta)$ the matrix corresponding to the permutation $1 \mapsto 2 \mapsto 3\mapsto 4 \mapsto 1$ and we extend $S$ by $0$ on all other points and arrows.  Then $S(M)$ is a
 brick and $Hom(S(M),S(M'))=0$
 holds for disjoint $M$ and $M'$. For the common dimension-vector we calculate $q_{Q}(\alpha)) \leq -16$.

If only $z$ belongs to $C$ we set  $S(y)=k$ and we take for $S(\beta)$ the matrix  having all entries equal to $1$. Again we extend this representation by $0$ outside the points and arrows 
already considered and  one obtains bricks with $Hom(S(M),S(M'))=0$ for disjoint $M,M'$. This time we get $q_{Q}(\alpha)) \leq -3$. The case where only $y$ belongs to $C$ can be treated 
similarly.

In both cases we find three disjoint subsets  and thus three orthogonal bricks.

Assume now that $Q$ contains no cycle i.e. it is a tree. Then it contains a tame tree $Q'$ and at least one additional arrow $\beta:y \rightarrow z$. We consider only the case that $z$ 
 belongs to $Q'$. The Auslander-Reiten quiver of $kQ'$ contains an indecomposable preprojective $U$ with $dim U(z) \geq 5$.  For $\lambda \in k$ we extend $U$ to a representation $S=S(\lambda)$
 by setting $S(y)=k$ and $S(t)=0$ on all other points. Furthermore $S(\beta)$ is defined by the column $c$ with $c_{1}=1,c_{2}= \lambda$ and $c_{i}=0$ for the remaining indices. Because $U$ is a brick 
this produces an infinite family of pairwise orthogonal bricks. Finally we obtain $q_{Q}(\alpha)) \leq -3$ because the dimension vector of $U$ is a real root.

\end{proof}

\subsection{The proof of the simultaneous embedding result}

First we apply proposition 3 to the orthogonal bricks $U,V$ in $Q$ delivered by lemma 4 to obtain  a strict representation embedding $F:rep\, K_{3}\longrightarrow  rep\, Q$ such that the dimension
 vectors of all $FX$ are multiples of a fixed root $\alpha$ with $q_{Q}(\alpha) \leq -3$. For a fixed $n$ we look at all $X$ such that the dimension-vector of $FX$ is $n\alpha$.
Then we always have $dim Hom (FX,U\oplus V) \geq 1$, but 
$Hom(W^{n},U\oplus V)$ vanishes for the $W$ from lemma 4.. By the semi-continuity of the dimensions of homomorphism-spaces the  orbits of all $FX$ lie in a proper closed subset.

So it remains to construct 'orthogonal' strict embeddings 
$E_{i}: mod \, A_{i} \longrightarrow rep \,K_{3}$ for all indices $i$ and to define $F_{i}=F \circ E_{i}$. Each $A_{i}$ is a homomorphic image of $kL_{p(i)}$ where  $p(i)\geq 0$ is the minimal number of algebra generators. 
 This gives a full embedding $mod \,A_{i} \longrightarrow rep \,L_{p(i)}$ that
 we compose with the strict embedding $F_{p(i)}$ of section  2.2 to find a strict embedding $G_{i}: mod\,A_{i}  \longrightarrow rep \,K_{p(i)+1}$. Since $k$ is infinite it has the same cardinality as the 
multiplicative group $k^{*}$.  We identify $I$ with a subset of $k^{*}$. 

On the other hand we get from lemma 3 and the embedding of section 2.2 a large set $Z$ of pairwise orthogonal bricks in $rep \,K_{3}$ and $Z$  is the disjoint union of the subsets $Z(q)$ 
consisting of the bricks
 with dimension vector $(q,q)$ for $q \geq 1$. All sets $Z(q)$ have the cardinality of $k$ and so they can be indexed by $k^{*}$. We want to find pairwise disjoint subsets $Z(q,i)$ 
with two elements for all
$ i \in k^{*}$. 
Thus we look at the multiplikative homomorphism $\phi:k^{*} \longrightarrow k^{*}$ defined by $\phi(x)=x^{m}$ where $m=2$ if the characteristic is not $2$ and $m=3$ otherwise. In each case $\phi$
 is surjective and has fibres of cardinality $m$. So we find the wanted disjoint subsets in both cases inside the fibres. 

Given $i \in I$ we consider the two orthogonal bricks $U,V$ in $Z(p(i)+1,i)$. Since $q_{K_{3}}(p(i)+1,p(i)+1)=-(p(i)+1)^{2}\leq -(p(i)+1)$ we find with proposition 3 a strict representation 
embedding 
$H_{i}: rep \, K_{p(i)+1} \longrightarrow rep\, K_{3}$. 

Now consider the strict representation embeddings $E_{i}= H_{i}\circ G_{i}$. By construction all the sets $Z(q,i)$ are pairwise disjoint and therefore all these embeddings are 
'orthogonal' to each other. The proof is complete.
 
 \subsection{Two amplifications}

 This proof is somewhat artificial because we use only part of the information available. It is more natural to consider the full subcategory $\mathcal{E}$ consisting of all
 representations having a finite filtration with subquotients in the above set 
 $Z$ of pairwise orthogonal bricks.  Ringels simplification \cite{RT} says that this is an abelian $k$-linear category with the modules in $Z$ as simple objects. Since all modules have finite
 length we can apply an old result of Gabriel \cite{GCA,GI2} and we get that $\mathcal{E}$ is equivalent to the category $rep^{nil}\,K$ of those representations of the quiver $K$ of 
$\mathcal{E}$ which are 
annihilated by some power of the ideal generated by the arrows. Now $K$  
 is easy
 to determine. Let $M$ be a point in $K$ i.e. a module in $Z$. Its dimension-vector is $m\alpha$ for some natural number $m$.  Then we have 
$m^{2}p+1$ 
loops in $Z$ because we have $q_{Q}(m\alpha)= dim Hom(M,M)-dim Ext(M,M)$. For another point $N$ in $K$ with dimension-vector $n\alpha$ there are by the same argument $mnp$ arrows between the two points 
in each direction. Recall also that for each $m$ the set of modules in 
$Z$ with dimension vector $m\alpha$ has the cardinality of $k$. Thus $rep^{nil}\,K$ is really large and it  contains in particular all $mod \, A_{i}$ of the theorem as pairwise orthogonal
 subcategories.

Another point is that the root $\alpha$  constructed above in lemma 3 lives  only on a  ( eventually small ) minimal wild subquiver $Q'$ of $Q$ and so does $\mathcal{E}$.
This restriction on $\alpha$ can be removed by using the following 
non-elementary result of Kac: Let $\alpha$  be a root with $q_{Q}(\alpha ) \leq -1$ and $(\alpha,e_{i}) \leq 0$ for all simple
 roots $e_{i}$. Then the generic representation in the variety $rep(\alpha)$ of representations of $Q$ with dimension vector  $\alpha$ is a brick and
the set of these bricks depends on $1-q(\alpha)$ parameters  ( see \cite{K} ): For a fixed 
brick $U$ one sees easily by considering an appropriate vector-bundle as in \cite{BoGeo} that $Hom(U,V)$ and $Hom(V,U)$ both vanish generically on $rep(\alpha)$. Using this repeatedly one 
finds an arbitrary finite number of pairwise orthogonal bricks. 
To satisfy also the condition $q_{Q}(\alpha ) \leq -3$  formulated in lemma 3 one has to replace eventually $\alpha$ by $2 \alpha$.
In the resulting generalization  the 'difficult' computational lemma 3 is no longer needed.

\section{Representation embeddings and the second Brauer-thrall conjecture}
\subsection{The main statements  and some comments}
In his  old report \cite{RBT}   on the Brauer-Thrall conjectures published 1980 Ringel formulated the following statement  ( without using the terminus  'representation embedding' )
and called it the
 'theorem of Nazarova and Roiter'.
 
\begin{theorem}Let $A$ be a finite-dimensional algebra having infinitely many isomorphism classes of indecomposable modules of finite dimension. Then there is a representation embedding
$F: mod \,k[T] \longrightarrow mod \,A$ mapping all simple $k[T]$-modules to modules of the same dimension.\end{theorem}

The truth of  the second Brauer-Thrall conjecture BT2 follows immediately,
 but later on it was recognized that the proof of the theorem given by Nazarova and Roiter in the long preprint \cite{NR} contains some errors and gaps ( see \cite{RB}    ) and therefore
 BT2 was still an 
open problem.

 The conjecture  was solved affirmatively 
around 1984 first by Bautista and then also by Fischbacher and Bretscher/Todorov ( see \cite{BBT2,BT,F} ).  All these proofs are  based on  multiplicative bases, on ray categories and their coverings and on 
the structure of large indecomposables over simply connected algebras 
( see \cite{BMB,F,BoTr} or \cite{BOsurvey} for a survey on all this material ).  Thus these proofs use  an approach which is completely different from 
the one of Nazarova and Roiter who played however a decisive role also in the new approach with his preprint \cite{RBo} containing again some errors and gaps but leading  to 
the central article \cite{BMB}
on multiplicative bases.

 In the time between the publication of the two preprints \cite{NR} and \cite{RBo} Roiter developed partly in collaboration with  Kleiner a completely new language to formalize the striking algorithms
 due to the Kiev-school: representations of bocses ( see \cite{Roiter,RBocs} ). This led Drozd to his fundamental 'tame or wild' theorem \cite{DrozdTW} proven without using any 
classification results but not to a proof
 of BT2 which is despite some serious efforts not yet proven in a purely conceptual way.

The theorem given above follows at least for a localization $k[T]_{f}$ of $k[T]$ easily by combining the truth of BT2 and Drozds theorem
and so the existence of $F$ is shown indirectly.

Here we  prove a stronger result in a constructive way without using Drozds theorem. Recall two definitions made before:

$ rep'\, K_{2}$ consists of all  Kronecker modules 
not containing the 
simple injective as a direct summand and $G:mod \,k[T] \longrightarrow rep' \, K_{2}$ is the obvious strict representation embedding $(V,T) \mapsto (V,V;T,id_{V})$.

\begin{theorem} Let $A$ be a representation-infinite  algebra of finite dimension. Then there is a representation embedding $F:rep'\,K_{2} \longrightarrow mod \, A$ with the 
following properties:

\begin{enumerate}
\item The representation embedding $H=F\circ G: mod \, k[T] \longrightarrow mod \, A$ is  given by tensoring with a bimodule $_{A}M_{k[T]}$ which is free of finite rank over $k[T]$.
\item If $A$ is basic of dimension $d$ then $_{A}M_{k[T]}$ is affine  of rank at most $max(30,4d)$.
\end{enumerate}

\end{theorem}

Part 1 implies immediately theorem 5 and part 2 supersedes my previous version of BT2 given in \cite[theorem 7.7]{BOsurvey}. Because the regular indecomposable representations all belong 
to $rep'\, K_{2}$ the theorem 
explains very well why the indecomposables occur in ${\bf P}^{1}(k)$-families. 

The theorem as well as its proof remain valid if $A$ is a $k$-split  algebra.

In a certain sense the theorem ist best possible:  We will see in the last section that  in characteristic $0 $  there is no finite-dimensional algebra $B$ with the property 
that $A$ is representation-infinite iff there is a representation embedding $F:mod \, B \longrightarrow mod \,A$. The classical Kronecker-algebra $kK_{2}$ satisfies this condition almost.

One can assume that $A$ is basic and minimal representation-infinite and one distinguishes two cases. If $A$ is not distributive 
we can apply corollary 2. If $A$ is distributive we use  the machinery and the results developed for the study of mild algebras including ray categories and their coverings
 ( see \cite{GR,BOsurvey} ). We do so not to make
 the paper unreadable but to make it rigorous. This is not always guaranteed in the literature about  this type of questions. The proof is independent of any lists obtained by computer calculations.

\subsection{The proof for non-distributive algebras}
To reduce to the case where $A$ is basic and minimal representation-infinite
we consider the  basic  algebra $B$ that is Morita-equivalent to $A$ and then a minimal representation-infinite residue algebra $B'=B/I$ thereof. The first 
representation embedding $R:mod \, B' \longrightarrow mod \,B $ is induced by tensoring with $B/I$. Thus, by lemma 2, part ii) holds for $B$ if it holds for $B'$.  Proposition 1 says that 
the composition of $R$ 
with the equivalence $ mod \,B \longrightarrow mod \,A$ is given by tensoring with a
 bimodule $_{A}M_{B'}$ which is projective over $B'$. 
Lemma 2 implies that the composition
$mod \,k[T] \longrightarrow mod \,B' \longrightarrow mod \,A$ is given by tensoring with a bimodule which is free over $k[T]$. Thus from now on we can always assume that $A$ is basic and 
minimal representation-infinite and so it is given by quiver and relations i.e. $A=kQ/I$.

Suppose now that $A$ is not distributive. Then corollary 2 implies part ii) of theorem 6. The rank is bounded by the dimension of $A$ by part iv) of theorem 1.

\subsection{A lemma about tame concealed algebras}

The next result is used for distributive algebras. It can be verified  with the help of a well-known list, but we include a theoretical argument.
\begin{lemma}
 Let $A=kQ/I$  be  a tame concealed algebra and let $N$ be a homogeneous simple regular module. Then there is -up to duality - a simple projective $S=S(x)$ with $dim\,N(x)=1$ and 
 one has an exact sequence
$0 \longrightarrow S \longrightarrow N \longrightarrow V \longrightarrow 0$ such that $S$ and $V$ are orthogonal bricks with $dim \,Ext(V,S) = 2$. 
\end{lemma}
\begin{proof} First of all we want to find a source or a sink $x$ in $Q$ with $dim \,N(x)=1$. This is obvious if $N$ is a quotient of an indecomposable projective. 
Thus we can assume that this is not true.

By well-known properties of tame concealed algebras \cite{ RT,SS}  the Euler-form  $q_{A}$ is critical and its radical is generated by the dimension-vector of $N$. A result of Ovsienko 
which is explained and proven with all details on the pages  314-315 of \cite{vH} enshures that at least
$dim\,N(x)=1$ for some point $x$. Assume that this never occurs for a sink or a source. Then we find in $Q$ a path $w= x_{0} \rightarrow x_{1} \rightarrow \ldots x_{i}=x  \ldots \rightarrow x_{p}$
with $0 < i < p$ and with $dim\,N(x_{0})>1<dim \,N(x_{p})$ but $dim \, N(x_{j})=1$ for  $1\leq j \leq p-1$.

 Now let $v$ be any path in $Q$ of length $\geq 1$ from $a$ to $b$ which is not in $I$.
 This induces a non-zero non-invertible homomorphism $P(b) \longrightarrow P(a)$ with indecomposable cokernel $C$ which is not isomorphic to $N$ since we have this case already settled.
Now the argument of Vossieck in \cite{V,RT}   says that the kernel of $N(v)$ is isomorphic to $Hom(C,N)$ and the cokernel to $DHom(N,DTrC)$. 
One of them is zero because otherwise $C$ is
preprojective as well as preinjective. Therefore $N(v)$ is mono or epi. In particular all arrows are represented in $N$ by a mono or an epi. We get that $N(w)$ is the composition of a 
proper epi followed by a proper mono and therefore non-zero but neither mono nor epi. Thus, up to duality, there is a sink $x$ with $dim \, N(x)=1$.

Let us look at the obvious exact sequence $0 \longrightarrow S=S(x) \longrightarrow N \longrightarrow V \longrightarrow 0$. Here $V$ is preinjective.
 Now $N\simeq DTrN$ has projective dimension and injective dimenion $ 1$ so that $[N,X]^{1}=[X,DTrN]=[X,N]$ and similarly $[X,N]^{1}=[N,X]$ holds for all $X$. Here and in the following $[X,Y]$
 resp. $[X,Y]^{1}$ denote the dimensions of $Hom(X,Y)$ resp. $Ext(X,Y)$.
From 
 $$0=[N,N]-[N,N]^{1}=[N,S\oplus V]-[N,S\oplus V]^{1}=[N,V]-[S,DTrN]=[N,V]-1$$ we get $1=[N,V]$. Therefore $V$ is indecomposable and $S$ and $V$ are orthogonal bricks. Finally
we apply $Hom(V, \,\,)$ to the exact sequence and we 
find $[V,S]^{1}=[V,V]+[V,N]^{1}=2$.

\end{proof}
\subsection{The proof for distributive algebras}

For distributive algebras we can even prove the following using some deep results:

\begin{theorem} Let $A$ be a distributive basic minimal representation-infinite 
algebra of dimension $d$. Then there is a representation embedding $$F: rep \,K_{2} \longrightarrow mod \, A$$ 
such that the composition $H=F\circ G$ is  induced  by an affine bimodule $_{A}M_{k[T]}$  of rank at 
most $max(30,4d)$.
 
\end{theorem}
\begin{proof}

By theorem 2 of \cite{BOIND} the algebra 
 $A$ is isomorphic to the linearization of its ray category $P$ which in turn has a universal covering  $\pi:\tilde{P }\longrightarrow P$ where $\tilde{P}$ is interval-finite 
and the fundamental group $G$ is
 free. Then by \cite{GABU} $\tilde{P}$ is not locally representation-finite and we distinguish two cases as in \cite{BoCr}. Either any finite convex subcategory is representation-finite 
or not. In the second case we choose a convex subcategory $C$ which is not representation-finite with minimal number of points. Then $C$ defines  a critical algebra and so a tame concealed 
algebra by theorem 22 in \cite{BOsurvey}.

First assume that $\tilde{P}$ contains a finite convex subcategory $C$ such that $B=kC$ is a tame concealed algebra of type $\tilde{E}_{n}$ for $6 \leq n \leq 8$ or $\tilde{D}_{m}$ for 
$4 \leq m \leq 2d + 1$. By lemma 4 there is up to duality a pair $S,V$ of orthogonal bricks with simple $S$ and with $dim \, Ext(V,S)=2$. Theorem 3 gives us a strict embedding
 $L:rep\,K_{2} \longrightarrow mod \,B$ such that the composition with $F_{1}:mod \,k[T] \longrightarrow rep \,K_{2}$ is induced by an affine bimodule of rank $dim (S\oplus V)$ which is bounded by
$max(30,4d)$. 

We claim that the composition $F$  of $L$ with the restriction $Q:mod \,B \longrightarrow mod\,A$ of the push-down functor is the wanted functor. It is exact as the composition of two exact functors and 
it maps indecomposables 
to indecomposables since $L$ and $Q$ do. Thus let $X$ and $Y$ be non-isomorphic  indecomposables. Then $LX$ and $LY$ are non-isomorphic indecomposables in the essential image of $L$ in which $S$ and $V$ are the only indecomposables
 with 
support different from $B$.  If $QLX$ and $QLY$ are isomorphic  then  
$LX$ is conjugate to $LY$  under some $g\neq 1$ in the fundamental group and so their supports have the same cardinality. It follows that $B$ is the support of both modules and so $g$ leaves
 a finite set of points in $\tilde{P}$ invariant. Thus $g$ 
has finite order which is impossible in a free group. Here we have used several times Gabriels results from \cite{GABU}.

Now $Q$ is in general not a representation embedding but it is exact. By the observation of Eilenberg and Watts it is therefore given by the tensor-product with a bimodule $_{A}Q_{B}$ which 
is projective over $B$. Since $Q$ preserves the dimensions of all simples it is in fact free of rank 1 by part iv) of proposition 1.
Lemma 2 shows that
$F\circ G=
Q\circ L \circ F_{1}$ is induced by an affine bimodule of rank bounded by $max(30,4d)$.

If the assumptions of the first case are not satisfied then $\tilde{P}$ contains a tilted algebra of type $\tilde{D}_{m}$ for some $m \geq 2d+2 $ or sincere representation-finite algebras
 with 
arbitrarily many simples as  convex subcategories. In both cases $\tilde{P}$ contains a line of length $2d$ as a convex subcategory. Here and in the following text we freely use the definitions  
and results from
 section 7.1 in \cite{BOsurvey} especially 
 lemma 7.2 ( where $d$ should denote the dimension of $kP$, not the number of simples   ). One finds a 'subline'
  $ q \rightarrow z_{1} \ldots  \rightarrow z_{t-1} \rightarrow z_{t} \leftarrow \ldots  x' \leftarrow q'$ with 
$\pi(q)=\pi(q')$ and such that the two outer arrows $\alpha$ and $\beta$ satisfy $\pi(\alpha) \neq  \pi(\beta)$. If we choose such a subline of minimal length then all arrows $\gamma$
 occurring in the subline and pointing from the left to the right satisfy $\pi(\gamma) \neq \pi(\alpha)$. Thus the subline is mapped by $\pi$ to a closed walk $w$ in $P$ that is not a power
 of another walk. 
 
On the other hand the subline can be prolonged on both sides to an infinite periodic line and it follows that $A$ is a zero-relation algebra and no subpath of the subline is a zero-relation. 
In fact, $A$ is
is special biserial as shown by Ringel in \cite{RM}.  

Now let $B$ be the path-algebra of the quiver of type $\tilde{A}_{e}$ obtained by identifying the end-points $q,q'$ of our subline. Thus $B$ is again special biserial with essentially only
 one closed
 walk $w'$. We have an obvious 'push-down' functor $H:mod \,B \longrightarrow mod \,kP$ which is a representation embedding as one easily checks by using the fact that all indecomposables 
upstairs and downstairs are given
 by string and band-modules and that $w=\pi(w')$ is not a power. As in the first case one sees that $H$ is induced by tensoring with a bimodule $_{A}H_{B}$ which is free of rank $1$ over $B$. Our wanted functor $F$ is the 
composition of $H$ with the obvious embedding 
$L:rep \,K_{2} \longrightarrow mod \,B$ that identifies $rep \,K_{2}$ with the full subcategory of $mod\,B$ where all arrows except $\alpha$ and $\beta$ are represented by isomorphisms.
 One sees that $L\circ G$ is induced by an affine bimodule of rank at most $2d$ and so is $F\circ G$.

\end{proof}

\section{About  embedding classes and the partial orders between them}

\subsection{Representation-finite algebras and wild algebras}

In the representation-finite case the description of the embedding classes is easy.
\begin{proposition} For two representation-finite algebras $A$ and $B$ the following conditions are equivalent:
\begin{enumerate}
 \item $A$ and $B$ are embedding equivalent.
\item $A$ and $B$ are strictly embedding equivalent.
\item $A$ and $B$ are Morita-equivalent.
\end{enumerate}
Furthermore in that case any representation embedding $F: mod \,A \longrightarrow mod \,B$ is an equivalence.

\end{proposition}
\begin{proof} Since $A$ and $B$ are embedding equivalent they have the same number $n$ of isomorphism classes of indecomposables. Let $F:mod \,A \longrightarrow mod \, B$ and 
$G: mod \,B \longrightarrow mod \,A$ be representation embeddings. We choose a list $U_{1}, \ldots , U_{n}$ of the indecomposable $A$-modules. Then $FU_{1},\ldots ,FU_{n}$ is a list 
of indecomposable $B$-modules and $GFU_{1}, \ldots ,GFU_{n}$ is again a list of indecomposable $A$-modules. Since $F$ and $G$ are faithful we get $$\sum_{i,j} dim\,Hom(U_{i},U_{j}) \leq \sum_{i,j} dim\,Hom(FU_{i},FU_{j}) \leq\sum_{i,j} dim\,Hom(GFU_{i},GFU_{j})$$ with equality of the end terms. It follows  that $F$ is an equivalence.
 
\end{proof}
Thus the embedding classes are as small as they can be in the representation-finite case.

In the other extreme we know already that all wild algebras are embedding equivalent and that all wild quiver algebras are even strictly embedding equivalent. 
The next result implies that there are many strict embedding classes of wild algebras because there are families of wild local algebras depending on arbitrarily many parameters.
\begin{proposition}
 Let $F:mod \,A \longrightarrow mod \,B$ be a strict representation embedding with $A$ and $B$ basic finite dimensional and $B$ local.  Then $A$ is local too. Furthermore $F$ is an equivalence
if $A$ and $B$  are strictly representation equivalent.
\end{proposition}
\begin{proof}
  Let  $T$ be the only simple  $B$-module. The chain of non-zero 
homomorphisms $U \rightarrow U/rad U \rightarrow T \rightarrow soc U \rightarrow U$ shows that $T$ is the only $B$ -module with one dimensional endomorphism algebra. Thus we have 
$FS\simeq T$ 
for any simple $A$-module $S$ and so there is only one up to isomorphism. From $Hom_{A}(A,S)\simeq Hom_{B}(FA,T)\simeq k$ we see that $FA$ is local whence a quotient of $B$ and so 
$dim B\geq dim FA=dimA$. If there is also a strict representation embedding in the other direction we obtain $dimA=dimB$ whence $FA=B$. Then $F$ is am equivalence by proposition 1.

\end{proof}

For arbitrary algebras there is not much known about the partial orders $\leq$ and $\leq_{s}$ except that the class of $k$ is for both orderings the smallest class whereas the class of 
$kK_{3}$ is  the biggest. Also the rough subdivision representation-finite/tame polynomial growth/tame exponential growth/wild is preserved.

\subsection{Tame concealed algebras}
 The next proposition gives some results for tame concealed algebras which we always suppose to  be basic.

\begin{proposition}Let $B$ be tame concealed  and $A$ basic representation-infinite. let $F:mod \, A \longrightarrow mod\,B$ be a representation embedding. Then the following holds:
\begin{enumerate}
                    \item Considering $A$ as a finite $k$-category let $A'$ be the full subcategory supported by all points $x$ such that $FP(x)$ is preprojective. 
Then $A'$ has an infinite preprojective component.
\item If $A$ is minimal representation-infinite it is tame concealed. Then the regular simples  $E_{1},E_{2},\ldots E_{n}$ in a tube $T$ of rank $n$ in $mod\,A$ are mapped into a tube 
 of rank $m\geq n$.
\item If $F$ is a strict embedding and $A$ is tame concealed then  different tubes are mapped into different tubes, the tubular type can only increase and for the null-roots  of $A$ and $B$
 we have $\mid n_{A}\mid \leq \mid n_{B} \mid$. Here $\mid x \mid $ denotes the sum of all components of the vector $x$.
\item If $F$ is strict, $A$ is tame concealed and $A$ and $B$ are strictly embedding equivalent then $F$ is an equivalence. 
                   \end{enumerate}

\end{proposition}
\begin{proof} Let $x$ be a point such that $FP(x)$ is not preprojective. Then for each $d$ there is only a finite number of indecomposable $B$-modules of dimension $\leq d$ with $Hom(FP(x),V)\neq 0$. 
Since $F$ is faithful and does not decrease the total dimension one has $U(x)\neq 0$ only for finitely many indecomposable $A$-modules of dimension $\leq d$. But by BT 2  there exists 
such a $d$ 
with infinitely many indecomposables of smaller dimension. Since $A$ has only finitely many points, $A'$ is representation-infinite. If $x$ is in $A'$ and $x \rightarrow y$ an arrow 
in the quiver of $A$ then $y$ is also in $A'$. We identify the $A'$-modules with the $A$-modules vanishing outside $A'$. Then the projective $P(x)$ to a point $x$ in $A'$ is also projective 
in $A$ and such a projective has a bound on the length of chains of irreducible morphisms ending in $P(x)$. Let $C$ be the connected component of the Auslander-Reiten quiver of 
$A'$ containing  $P(x)$. We claim that $C$ is preprojective. First of all any module is the translate of a projective. This is clear for at least one module. Assume it holds for $U$  and let 
$U \rightarrow V$ or $V \rightarrow U$ be an irreducible map. If all $DTr^{i}V$ are defined there are  $n,n'$ and an irreducible map $DTr^{n'}V \rightarrow DTr^{n}U=P(y)$ and then
 $P(y)$ has no bound on the chains of irreducible maps ending there.  Thus $C$ consists only of translates of projectives. An oriented cycle in $C$ can be translated often enough 
to contain a projective and this is impossible again. Thus $C$ is a preprojective component. Since $A'$ is representation-infinite one of these components is infinite.

If $A$ is minimal representation-infinite we have $A=A'$ and there is only one preprojective component whence $A$ is tame concealed. Regular indecomposables lie on oriented cycles of irreducible
 maps and so they are mapped to regular modules. We can assume $n\geq 2$. There are $n$ isomorphism classes of indecomposables of regular 
length $n$ in the tube $T$ connected by a cycle of non-zero maps. These are mapped to $n$ non-isomorphic indecomposables of the same regular length in an a tube of rank $m$ and 
the claim follows.

It is clear that different tubes are mapped to different tubes if $F$ is strict.
Furthermore there is at least one homogeneous regular simple $N$ that is mapped to a homogeneous regular simple $N'$. Now let $S_{1},S_{2},  \ldots S_{r}$ be a list of simples in 
$mod \,A$ and let $n$ be the null-root of $A$. Using analogous notations for $B$ we find 
$$\sum _{i=1}^{r'}n'_{i} = dim N' = dim FN = \sum _{i=1}^{r} n_{i}dim FS_{i} \geq \sum _{i=1}^{r}n_{i}.$$ A strict representation embedding $G: mod \,B \longrightarrow mod \,A
$ gives the reversed inequality and it follows that $F$ and $G$ map simples to simples whence $r=r'$ and after renumbering $FS_{i}=S'_{i}$. For the projective cover $P_{i}$ of 
$S_{i}$ it follows from $Hom_{A}(P_{i},S_{j}) \simeq Hom_{B}(FP_{i},S'_{j})$ for all $j$ that $FP_{i}$ is a homomorphic image of the projective cover $P'_{i}$ of $S'_{i}$. Thus we find
$dim\,B \geq dim \,A$. By symmetry we have equality and we obtain $FP_{i}\simeq P'_{i}$. This holds for all indices whence $FA\simeq B$. By the remark of Eilenberg and Watts $F$ 
is a Morita-equivalence.

\end{proof}

It is easy to give examples of tame concealed algebras with $A\leq_{s} B$ that are not induced by an inclusion of the quivers
 but the complete partial order $\leq_{s}$ is
 not clear.
\subsection{Self-embeddings of the  classical Kronecker-modules}

For any finitely generated algebra $A$ one has the three monoids $$M(A) \subseteq S(A) \subseteq E(A)$$ consisting of the endo-functors $F: mod\,A \longrightarrow mod \,A$ that 
are equivalences resp. strict representation embeddings resp. representation embeddings. Our previous results show $M(A)=E(A)$ for all representation-finite algebras and $M(A)=S(A)$ for 
all tame concealed algebras as well as for all  wild local algebras. It is easy to see that for any wild algebra $A$ the two monoids $M(A)$ and $E(A)$ are different and that $E(A)$ is huge.
Here in this section we describe $E(kK_{2})$  and we apply the result in the last section to determine the $\leq$-minimal embedding classes.
 The results depend on the characteristic of $k$.

We need some explicit calculations with representations of $K_{2}$ and we fix some notations. In $kK_{2}$ we take as a basis the two idempotents $e_{1}$ and $e_{2}$ and the two arrows $\lambda,\rho$
 from 2 to 1. For the indecomposable  $P(i)$ with dimension-vector $(i+1,i)$  we take at the point 2 a vectorspace with basis $x_{1},x_{2},\ldots, x_{i}$ and 
at the point 1 
a vectorspace with basis $y_{1},y_{2},\ldots ,y_{i+1} $ and  and we define  $\lambda (x_{j})=y_{j}$ and $\rho (x_{j})=y_{j+1}$ for all $j$. For $P(0)$ we choose a base vector $z$. 
In $Hom(P(0),P(1))$ 
we take the basis $l,r$ defined by $lz=y_{1}$ and $rz=y_{2}$. The indecomposable preinjective with dimension-vector $(i,i+1)$ is denoted by $I(i)$.

For $rep\,K_{2}$ there are only two reflection functors $S^{+}$ resp. $S^{-}$ which both are endo-functors that fix the regular indecomposables and act on the preprojective or preinjective indecomposables
 as shifts to the left resp. to the right. Here one has $S^{+}P(0)=0$ and $S^{-}I(0)=0$. We denote for arbitrary $i,j \geq 0$ by $rep\,K_{2}(i,j)$ the full subcategory of $rep\,K_{2}$ 
consisting of all objects that do not contain one of the modules $P(l)$ with $l<i$ or $I(m)$ with $m<j$ as a direct summand. Then $S^{+}$ induces for $i>0$ and arbitrary $j$ an equivalence between 
$rep\,K_{2}(i,j)$ and $rep\,K_{2}(i-1,j+1)$ and there is an analogous statement for $S^{-}$.

The following two lemmata about the existence and uniqueness of self-embeddings of $rep \,K_{2}$ are of independent interest. In the proofs we are working with companion matrices occurring e.g.
in the rational normal form RNF of a matrix. Recall 
that the RNF to a Jordan bloc of size n with eigenvalue $x$ is the companion matrix
$B((X-x)^{n})$.
A  short algorithmic treatment of the RNF which is suitable for beginners  is given in \cite{BRNF}.

\begin{lemma} Let $n\geq 1$ be a natural number. 
\begin{enumerate} 
 \item There is an exact functor $F_{n}:rep\,K_{2} \longrightarrow rep \,K_{2}$ with $FP(0)=P(0)$, $FP(1)=P(n),F(l )(z)=y_{1},F(r)(z)=y_{n+1}$.
\item  $F_{n}$ reflects the isomorphism classes.
\item $F_{n}$ preserves indecomposability iff $n=1$ or $ char\,k=p>0$ and $n=p^{m}$ for some $m\geq 1$.
\end{enumerate}
\end{lemma}
\begin{proof} We construct for $A=B=kK_{2}$ a bimodule $_{B}M_{A}$ with $_{B}M_{A}\otimes \,-\simeq F$. As a right-module we take $M_{A}= e_{1}A \oplus ( e_{2}A) ^{2n-1}$. The multiplication of
$e_{1},e_{2},\lambda, \rho$ from the left 
 is given by the following elements $E_{1},E_{2},S,R$ in $End (M_{A})$ that we describe by their action on the obvious $k$-basis $B$ consisting of $e_{1},\lambda, \rho, e_{2}(j),1\leq j \leq 2n-1$ of $M_{A}$.
We set $E_{2}(e_{2}(2j-1))=e_{2}(2j-1)$ for $1\leq j \leq n$ and $E_{2}(x)=0$ for the other base vectors of $M_{A}$. The identity is the sum of $E_{1}$ and $E_{2}$. We define
$S(e_{2}(1))=\lambda, S(e_{2}(2j-1))=e_{2}(2j-2)$ for $2\leq j \leq n$ and $S(x)=0$ for the other base vectors.  Finally we have $R(e_{2}(2j-1))=e_{2}(2j)$ for $ 1 \leq j \leq n-1$,
 $R(e_{2}(2n-1))=\rho$ and $R(x)=0$ for the other base vectors. Then $FP(0)=P(0)$ is obvious and  $M_{A}\otimes Ae_{2}$ has as a $k$-basis the $x\otimes e_{2}$ with $e_{1}\neq x \in B$. With respect to 
this basis one verifies $FP(1)=P(n)$ and the assertion about $F(l)$ and $F(r)$ immediately.

Next we determine $FU$ for any indecomposable $U$.  For the normed polynomial $Q=a_{0} + a_{1}X + \ldots +a_{q-1}X^{q-1}+ X^{q}$  we denote   by $L=L(Q)$ the representation with $L(1)=L(2)=k^{q}$
 where $\lambda$ acts by the identity matrix and $\rho$ by the companion matrix $B(Q)$ of $Q$. 
define  $\alpha:P(0)^{q} \rightarrow P(1)^{q}$ by the matrix 
\[\left [
\begin{array}{cccccc}
r&0&0& ...&0&a_{0}l\\
-l&r&0&... & 0&a_{1}l\\
0&-l&r&...&0&a_{2}l\\
...&...&...&...&...&...\\
...&...&...&...&r&a_{q-2}l\\
...&...&...&0&-l&r+a_{q-1}l\\

\end{array}
\right ].
\]

Then the cokernel has the residue classes of the $x_{1}(j),y_{1}(j)$ ,$1 \leq j \leq q$ as a basis and with respect to this basis $\lambda$ is the identity and $\rho$ is $B(Q)$.
Then $FL(Q)$ is isomorphic to the cokernel of $F(\alpha)$ which is isomorphic to $L(Q(X^{n}))$ as one sees by choosing in the cokernel the residues of $x_{j}(i)$, $1 \leq j \leq n$, 
$1 \leq i \leq q$ and their images under $\lambda$ as a basis. For the indecomposable $L((X-x)^{k}))$ we get $L((X^{n}-x)^{k}))$. This is indecomposable for $x=0$. For $x\neq 0$ we write
 $n=p^{m}n'$ where 
$p$ is the characteristic of $k$ if this is positive or $p=1$ and $p$ and $n'$ are coprime. Let $x'$ be the unique scalar with $(x')^{p^{m}}=x$ and let $(x'')^{n'}$ be $x'$. 
Finally let $\zeta$
 be a generator for the group of all $y$ with $y^{n'}=1$.  Then we get $$(X^{n}-x)=(X^{n'}-x')^{p^{m}}= \prod_{i=1}^{n'}(X-\zeta^{i}x'')^{p^{m}}$$ and consequently
$L((X^{n}-x))^{k})= \oplus_{i=1}^{n'} L((X-\zeta^{i}x'')^{p^{m}k})$. The regular indecomposables where $\rho$ is the identity and $\lambda$ a nilpotent Jordan-bloc are mapped to indecomposables under $F$.
We conclude that for regular indecomposables $U\not\simeq V$ the modules $FU$ and $FV$ have only regular indecomposable direct summands but none in common. One verifies easily that
$FP(i)=P(ni)$ and $FI(i)=I(n(i+1)-1)$ hold for all $i$.

Thus $F$ reflects always the isomorphism classes and it preserves indecomposability only for $n=1$  or for $char \,k=p>0$ and $n=p^{m}$.

\end{proof}

\begin{lemma} Let $F:rep\,K_{2} \longrightarrow rep\,K_{2}$ be a representation embedding. Then we have:
 \begin{enumerate}
  \item $F$ is an equivalence for  $char\,k=0$.
\item For $char\,k=p>0$ we have $F=(S^{-})^{i}GF_{p^{m}}$. Here $G:rep \,K_{2} \longrightarrow rep\,K_{2}$  is an equivalence, $i < p^{m}$ and $F_{p^{m}}$ is the functor 
introduced in lemma 6. Here all regular indecomposables of regular length a multiple of $p^{m}$ lie in the essential image.
 \end{enumerate}

\end{lemma}

\begin{proof}
Set $X=FP(0)$,$Y=FP(1)$, $f=Fl$ and $g=Fr$. Since $Y$ has infinitely many pairwise non-isomorphic quotients and $Hom(X,Y) \neq 0$ we have $X=P(i)$ and $Y=P(i+n)$ for some $i$ and  $n\geq 1$. The inequality
 $i <  n$ holds because $X^{2}=FP(0)^{2}$ embeds into $Y=FP(1)$. Replacing $F$ by 
$(S^{+})^{i}F$ we can assume $X=P(0)$. For $n=1$ the functor $F$ is an equivalence by proposition 1. Thus we assume $n\geq 2$.
 For any $(a,b) \neq (0,0)$ the cokernel of $h=h(a,b)=af+bg$ is indecomposable regular. If there is an $h$ with 
$h(z)= p_{2}y_{2}+p_{3}y_{3}+ \ldots p_{n}y_{n}$
 then we have in the cokernel $ \lambda (p_{2}x_{2}+p_{3}x_{3}+ \ldots p_{n}x_{n})=0$ and $ \rho (p_{2}x_{1}+p_{3}x_{2}+ \ldots p_{n}x_{n-1})= 0$. But for a 
regular indecomposable $\lambda$ or $\rho$ is injective.

Now $Gl_{2}(k)$ acts in an obvious way by automorphisms on $kK_{2}$ inducing auto-equivalences of $rep\,K_{2}$.
Applying such an appropriate equivalence $G^{-1}$ we can arrange  that $f(z)= y_{1} + p_{2}y_{2}+ \ldots p_{n}y_{n}$ and  $g(z)= q_{2}y_{2}+q_{3}y_{3}+  \ldots y_{n+1}$ 
for appropriate scalars $p_{i},q_{j}$.

Then for all $bf+g$, $ b \in k$, we can take the residues 
of $x_{1}, \ldots x_{n}$ and of $y_{1},y_{2}, \ldots y_{n}$ as a basis of the cokernel. With respect to these bases $\lambda$ acts by the identity matrix and $\rho$ by a companion matrix 
$C$ with 
 $C_{i,n}=-(bp_{i}+  q_{i})$ for $2\leq i \leq n$ and $C_{1,n}= -b$. The cokernel is indecomposable iff $C$ is similar to a Jordan bloc with eigenvalue $x=x(b)$. Comparing the 
characteristic polynomials we find 
\[X^{n} + (bp_{n}+q_{n})X^{n-1} +  \ldots +(bp_{i+1}+q_{i+1})X^{i} + \ldots ...(bp_{2}+q_{2})X + b =(X-x)^{n}.\]
 For $char \,k =0$ we  compare the coefficients of $X^{n-1}$ and of $X^{0}$  and we obtain for all $b$ the relation $(\frac{bp_{n}+q_{n}}{n})^{n}=b$ which is impossible.
Thus we have $char\,k=p>0$ and we write $n=p^{m}n'$ with $n'$ coprime to $p$. We get $(X-x)^{n}=(X^{p^{m}}-x^{p^{m}})^{n'}$. For $n'>1$ one finds a contradiction as in characteristc 0 
by comparing now the coefficients of $X^{p^{m}(n'-1)}$ and $X^{0}$. So we have $n'=1$. Then $bp_{i+1}+q_{i+1}$ vanishes for all $b$ and all $1\leq i \leq n-1$ i.e. $f(z)=y_{1}$ 
and $g(z)=y_{n+1}$.
Altogether we find $G^{-1}(S^{+})^{i}F=F_{p^{m}}$ or 
$F=(S^{-})^{i}GF_{p^{m}}$. Reversely $GF_{p^{m}}$ maps $I(0)$ to $I(p^{m}-1)$ and so $(S^{-})^{i}GF_{p^{m}}$ is a representation embedding for all $i < p^{m}$. The last assertion is evident
from the proof of lemma 6.

\end{proof}

\subsection{The minimal  embedding classes of the representation-infinite algebras}

\begin{lemma}
 Let $A$ be a basic minimal representation-infinite algebra. Then there is a representation embedding $F:rep K_{2} \longrightarrow mod \,A$ or $A$ is isomorphic to $k[X,Y]/(X,Y)^{2}$.
\end{lemma}

\begin{proof}
 By theorem 7 we can assume that $A$ is not distributive. We consider $A$ as a finite category. First let $x$ be a point such that $A(x,x)$ is not uniserial. Since $A$ is minimal 
representation-infinite so is $A(x,x)$ and $A(x,x) \simeq k[X,Y]/(X,Y)^{2}$ follows. One has an exact sequence 
$$0 \longrightarrow S_{x}^{2} \longrightarrow P_{x} \longrightarrow V \longrightarrow 0$$ of $A$-modules where $U:=S_{x}$ resp. $P_{x}$ are the simple resp.the indecomposable projective
 to the point $x$. Here $U$ and $V$ are bricks and we have $dim\,Ext(V,U)=2$. For $Hom(U,V)=0$ we obtain the wanted functor $F$ by theorem 3. For $Hom(U,V)\neq 0$ we get $V=S_{x}$ and 
$A\simeq k[X,Y]/(X,Y)^{2}$.

Next assume that all $A(x,x)$ are uniserial but  some $A(x,y)$ is not a uniserial bimodule. Let $B$ be the full subcategory of $A$ supported by $x$ and $y$. The minimality of $A$ implies the
 minimality of $B$ and there are only two possibilities. Either $B$ is isomorphic to $kK_{2}$ or the quiver of $B$ has a loop $\alpha$ at $x$, a loop $\gamma$ at $y$ and an arrow $\beta$ from 
$x$ to $y$. The relations are given by $\alpha^{2},\gamma^{2}$ and $\gamma \beta \alpha $. 

In the first case we look at the exact sequence of $A$-modules 
$$0 \longrightarrow S_{y}^{2} \longrightarrow P_{x} \longrightarrow V \longrightarrow 0.$$ Here $U=S_{y}$ and $V$ are orthogonal bricks and so theorem 3 gives the wanted functor $F$.
The second case is more interesting. Choose some paths $a,b,c$ in the quiver of $A$ that induce $\alpha,\beta,\gamma$ in $B$. Then $ba$ and $cb$ are two paths delivering two linearly
 independent elements in the left and right socle of $A$.  Let $U$ denote the submodule of $P_{x}$ that is generated by $b$. Then there is an exact sequence
$$0 \longrightarrow S_{y}\oplus U \longrightarrow P_{x} \longrightarrow V \longrightarrow 0.$$

This induces an isomorphism $Hom(S_{y} \oplus U,U) \simeq Ext(V,U)\simeq k^{3}$. Now $U$ and $V$ satisfy $Hom(U,V)=Hom(V,U)=0$ and $End \,U \simeq End \,V \simeq k[X]/(X^{2})$. To get 
the wanted embedding $F$ 
by applying theorem 3 we have 
to verify that $dim Ext(V,U)/J$ has dimension 2. Here $J$ is the radical of the $End \,U-End \,V$-bimodule $Ext(V,U)$. We consider the exact restriction functor $G$ from $mod\,A$ to 
$mod \,B$. It induces the following diagram:

\setlength{\unitlength}{0.8cm}
\begin{picture}(9,3)
\put( 3,2.5){$ Hom_{A}(S_{y}\oplus U,U)$}  \put(9,2.5){  $Ext_{A}(V,U)$}\put(3,0.5){ $Hom_{B}(GS_{y}\oplus GU,GU)$} \put(9,0.5){ $ Ext_{B}(GV,GU)$}
\put(5,2){\vector(0,-1){0.8}} \put(10,2){\vector(0,-1){0.8}} 
\put(8,2.7){\vector(1,0){0.8}} \put(8,0.7){\vector(1,0){0.8}} 
\end{picture}

Here the horizontal maps are the connecting homomorphisms 
which are isomorphisms by construction. The vertical maps are induced by $G$ and so the left map is an isomorphism and the right map an $End\,U-End\,V$-bimodule homomorphism. 
Since the connecting
homomorphisms are given by push-outs the diagram commutes. Now for $B$ an easy but lengthy direct calculation with pushouts and pullbacks shows that $Ext_{B}(GV,GU)$
has one-dimensional radical as an $End_{B}\,GU -End_{B}\,GV$-bimodule. 

\end{proof}

The two categories $mod \,kK_{2}$ and $mod \,k[X,Y]/(X,Y)^{2}$ are stably equivalent and they are related by two well-known  $k$-linear functors  namely by
$$G:mod \,k[X,Y]/(X,Y)^{2} \longrightarrow mod\,kK_{2} \mbox{ given by}\, GM=(M/radM,rad M,x,y)$$ where $x$ and $y$ are
 induced by the multiplications with $X$ and $Y$ and by $$H:mod\,kK_{2} \longrightarrow mod \,k[X,Y]/(X,Y)^{2} \mbox{ defined by} \,H(V,W,\alpha,\beta)=V\oplus \,W$$ with $X$ acting by $\alpha$ and $Y$ acting by $\beta$.
These functors have the following properties that are easy to verify: $HGM\simeq M$ holds for all $M$ - and so each indecomposable $B$-module is hit by an indecomposable $A$-module - 
whereas $GHU \simeq U$ holds only for all indecomposables $U \not\simeq (0,k,0,0)$.
$H$ is exact whereas  an exact sequence $0 \rightarrow X \rightarrow Y \rightarrow Z \rightarrow 0$ is  mapped under $G$ to an exact sequence only if $rad X= X \cap rad Y$ or equivalently the only 
simple $B$-module $S$ 
is not a direct summand of $X$. Epimorphisms are always preserved by $G$. Thus both are almost representation embeddings, but $G$ is not exact and $H$ maps the two non-isomorphic 
simple modules of $kK_{2}$ to the simple  of $k[X,Y]/(X,Y)^{2}$. In the sequel we will use both functors.

\begin{lemma}
 Suppose that $k$ has characteristic 0 and that $A$ is a representation-infinite algebra. Then we have:
\begin{enumerate}
 \item Any representation embedding $F:mod \,A \longrightarrow mod\,kK_{2}$ is an equivalence.
\item No representation embedding $F:mod \,kK_{2} \longrightarrow mod \,k[X,Y]/(X,Y)^{2}$ exists.
\item Any representation embedding $F:mod \,A \longrightarrow mod\,k[X,Y]/(X,Y)^{2}$ is an equivalence.
\end{enumerate}

\end{lemma}

\begin{proof}  We consider a minimal representation-infinite quotient $B$ of $A$. By proposition 6,
 $B$ is tame concealed and it has only homogeneous tubes whence $B$ is
( Morita-equivalent to )
 $kK_{2}$ and we have  
 a representation embedding $$ E: mod \,B \longrightarrow mod\, A.$$ The composition $F\circ E$  
is an equivalence by lemma 7. It follows that $F$ induces a bijection on the isomorphism classes and on the homomorphism spaces whence it is an equivalence.

We prove the second part. So assume that there is a representation embedding $F$. If it does not hit  $S$ then $GF$ is a representation embedding whence an equivalence. This is impossible 
because $P(0)$ is not hit by $G$.
Therefore we have $FX\simeq S$ for some $X$. The exactness of $F$ shows that $X$ is simple. Up to duality we can assume that $FP(0)\simeq S $ holds.
 $FP(1)=HY$ has infinitely many non-isomorphic quotients whence $GHY\simeq Y$ also does and $Y=P(i)$ for some $i\geq 1$ follows. For $i=1$ the quotient $P(1)/(P(0)^{2})$ is also
 mapped to $S$ by $F$. For $i>1$ we have $Hom_{B}(FP(0),FP(1)) = Hom_{B}(HP(0),HP(i)) \simeq Hom_{A}(P(0),P(i))$ and we  can  argue as in  the proof of lemma 7 
to find a contradiction. Therefore  $F$ does not exist.

For the last assertion we look at a minimal representation-infinite quotient $B$  of the basic algebra to $A$. By lemma  8  there is a representation embedding 
$$E:mod \,C \longrightarrow mod \,B  \longrightarrow mod \,A$$ where $C=kK_{2}$ or $C=k[X,Y]/(X,Y)^{2}$. Looking at the composition $F \circ E$ we see from part ii) 
that only the second case is possible. 
If $F\circ E$ does not hit
 $S$ then $G\circ F \circ E$ is a representation embedding which is impossible. Thus $F\circ E$ hits $S$ and so it maps $k[X,Y]/(X,Y)^{2}$ to itself whence $F\circ E$ is an equivalence.
 Then $F$ is also surjective on objects and  homomorphism spaces whence an equivalence.

\end{proof}

\begin{lemma}
 Suppose that $k$ has positive characteristic $p$.Then we have:
\begin{enumerate}
 \item There is a representation embedding $F:mod \,kK_{2} \longrightarrow mod \,k[X,Y]/(X,Y)^{2}$.
\item Let $F:mod \,A \longrightarrow kK_{2}$ be a representation embedding with  $A$ of infinite representation type.
Then $A$ is Morita-equivalent to $kK_{2} \times B$ where $B$ is a representation-finite algebra. 
\item The set of Morita-equivalence classes of finite-dimensional algebras in the embedding class of $kK_{2}$ is countable.
\end{enumerate}

\end{lemma}
\begin{proof} For part i) we take the representation embedding  $$F_{p}: rep \,K_{2} \longrightarrow rep \,K_{2}$$ as in lemma 6. It does not hit the simple injective and the composition
 $F=H\circ F_{p}$ with the functor $H$  is a representation embedding as one sees easily.

In part ii) we  can assume that $A$ is basic and we consider it as a finite category. In the convex subcategory $A'$ introduced in part i) of  proposition 6 there is the connected component
 $B$ 
consisting of the points $x$ such that the correponding indecomposable $A'$-modules all belong to the infinite preprojective component $C$ given by part i) of propostion 6. Finally
 let $D$ be a 
full subcategory of $B$ of smallest cardinality such that $C$ contains infinitely many $B$-modules with support $D$. Then $D$ is connected convex and it has a preprojective component 
containing infinitely many  sincere indecomposables  whence $D$ is a representation-infinite tilted algebra. Extension by zero and the given $F$ give us a representation embedding

$$ mod \,D \longrightarrow mod \,A \longrightarrow mod \,kK_{2}$$ and so $D$ is tame concealed. By part ii)  of proposition 6 we find $D\simeq kK_{2}$ because $D$ has only homogeneous tubes.

 Let $x,y$ be the support of $D$ and let 
$\alpha$ and $\beta$ be the two arrows from $x$ to $y$. Suppose there is an arrow $\gamma:y \longrightarrow z$. Then $FP(z)$ is also preprojective and $\gamma \circ \alpha$,$\gamma\circ \beta$ 
are linearly independent because $F$ is faithful and because all non-zero homomorphisms between indecomposable preprojective Kronecker-modules are injective. We get the contradiction that $A$ 
is wild. The same contradiction would follow from an additional arrow starting in $x$ or an additional arrow ending in $y$. 

Now the injective $I(y)$ to the sink $y$ has infinitely many submodules  whence $F(I(y))$ belongs to the preinjective component of $kK_{2}$. Let $\delta:z \longrightarrow x$ be 
an arrow in the quiver of $A$. Then $FI(z)$ and $FI(x)$ are also preinjective and there is no relation on the full subquiver supported by $\alpha, \beta,\delta $ because all non-zero maps 
between indecomposable preinjectives in $rep \,K_{2}$ are surjective. Thus $A$ would be wild. Therefore the points $x,y$ define a connected component of $A$ and we obtain
$A\simeq kK_{2}\times B$. 

If $B$ is representation-infinite we find another copy of $kK_{2}$ as a direct factor of $B $ and $A$ and consequently two representation embeddings
 $$G_{i}:mod \,kK_{2}\longrightarrow mod \,kK_{2}$$ whose essential images have no indecomposable in common. But by part ii) of lemma 7 there are natural numbers $m_{i}$ such that the 
essential image of $G_{i}$ 
contains all regular indecomposables of regular length a multiple of $p^{m_{i}}$. In particular the essential images intersect. Thus $B$ is representation-finite.

Finally, by the existence of multiplicative bases for representation-finite algebras there are only countably many Morita-equivalence classes of representation-finite algebras and so it 
suffices
 to construct for any natural number $n$ an embedding of $kK_{2}\times k^{n}$ into $rep \,K_{2}$. One starts with the embedding $F_{p}:mod \,kK_{2} \longrightarrow mod \,kK_{2}$ that does not
 hit a regular simple and one emdeds the semisimple category $mod \,k^{n}$ by choosing $n$ regular simples as the images of the simples.
\end{proof}

As an easy consequence of the last two lemmata one gets that two algebras of the infinite family $k\langle X,Y \rangle/ (X^{2},Y^{2},XY-\lambda YX)$ are in the same embedding class iff 
they are isomorphic. Thus there are infinitely many embedding classes of tame algebras.

At the end we  summarize our results about minimal embedding classes. 
\begin{theorem}Let $A$ be a representation-infinite algebra. Then the following is true:
\begin{enumerate}
 \item For $char \,k=0$ we have $kK_{2} \leq A$ or $k[X,Y]/(X,Y)^{2} \leq A$. The embedding classes to $kK_{2}$ and $k[X,Y]/(X,Y)^{2}$ are minimal and both contain only one Morita-equivalence class.
\item For $char \,k=p>0$ we have $kK_{2} \leq A$. Thus there is only one minimal embedding class, but it contains countably many Morita-equivalence classes.
\end{enumerate}

\end{theorem}

\end{document}